\documentclass[onefignum,onetabnum]{siamart220329}
\usepackage{amsfonts,amsmath,amssymb,hyperref,xcolor}
\usepackage[final]{changes}

\title{Coarse grid corrections in Krylov subspace evaluations of the matrix exponential}

\author{%
M.A. Botchev%
\thanks{
Keldysh Institute of Applied Mathematics of 
Russian Academy of Sciences, Miusskaya~Sq.~4, Moscow 125047,
Russia, \email{botchev@kiam.ru}.}
}

\newcommand{\Cc}{\mathbb{C}}

\newcommand{\eps}{\epsilon}
\newcommand{\fs}{\footnotesize}
\newcommand{\geqs}{\geqslant}

\newcommand{\leqs}{\leqslant}

\newcommand{\phiRT}{\texttt{phiRT}}

\newcommand{\rcg}{r_{\mathrm{cg}}}
\newcommand{\Rr}{\mathbb{R}}
\newcommand{\Rrnn}{\mathbb{R}^{n\times n}}
\newcommand{\RrNN}{\mathbb{R}^{N\times N}}
\newcommand{\tol}{\mathtt{tol}}
\newcommand{\yex}{y_{\mathrm{ex}}}

\newcommand{\ymg}{y_{\mathrm{mg}}}
\newcommand{\yref}{\tilde{y}_{\mathrm{ref}}}
\newcommand{\ytex}{\tilde{y}_{\mathrm{ex}}}

\newtheorem{note}{Note}

\begin{document}
\maketitle

\centerline{\emph{Dedicated to Victor Timofeevich Zhukov 
  on the occasion of his 70th birthday}}

\begin{abstract}
A coarse grid correction (CGC) approach is proposed to enhance the efficiency
of the matrix exponential and $\varphi$ matrix function evaluations.
The approach is intended for iterative methods computing the
matrix-vector products with these functions. It is based
on splitting the vector by which the matrix function is
multiplied into a smooth part and a remaining part.
The smooth part is then handled on a coarser grid,
whereas the computations on the original grid are carried out with
a relaxed stopping criterion tolerance.  Estimates on the error
are derived for the two-grid and multigrid variants
of the proposed CGC algorithm.
Numerical experiments demonstrate the efficiency of the algorithm,
\added{when employed in combination with Krylov subspace and Chebyshev
polynomial expansion methods}. 
\end{abstract}



\section{Introduction}
This paper presents an approach to use spatial multigrid techniques for
computing matrix-vector products with the matrix exponential and
$\varphi$ matrix function.  This approach is intended for iterative
methods computing the matrix-vector products with these matrix functions,
which appear abundantly, for instance, in exponential time integration 
of spatially discretized PDEs~\cite{HochbruckOstermann2010}.

Our approach is based on splitting the vector, by which the matrix function has to be
multiplied, into a smooth part, which can be well represented on a coarser
spatial grid, and a remaining, nonsmooth part. 
Computational gain is then achieved, because the smooth part is handled
on a coarser grid, whereas only the remaining nonsmooth part is handled on the original
fine grid with, typically, a relaxed accuracy tolerance.
In fact, as will be clear, the smaller the nonsmooth part is in norm,
the more relaxed tolerance can be used.
To estimate the error caused by the coarse grid solution part, the exponential
residual concept is
used~\cite{CelledoniMoret97,DruskinGreenbaumKnizhnerman98,BGH13}.
For our two-grid method, we show that its error is bounded by the terms
whose norm is controlled by the tolerance
in both coarse and fine grid solvers plus a term proportional to
$$
\|(Q\tilde{A}- AQ)\ytex(t)\|.
$$
Here $Q$ is the coarse-to-fine grid prolongation operator, $A$ and $\tilde{A}$
are respectively the fine and coarse grid matrices and $\ytex(t)$ is
the exact coarse grid solution (more precisely, it is the matrix function
times the smooth coarse-grid part of the given vector).
Thus, the accuracy of our method is restricted and depends on how much
the fine and the coarse grid solutions differ.  Nevertheless, numerical experiments
show that for moderate accuracy requirements, typical for solving large scale PDEs,
our approach can be very efficient.
Furthermore, we propose a procedure to estimate the coarse grid correction
error $\|(Q\tilde{A}- AQ)\ytex(t)\|$ in practice.

Multigrid techniques have been applied to the solution of time-dependent problems
since the appearance of multigrid.
The work of R.P.~Fedorenko~\cite{Fedorenko1961,Fedorenko1964},
the first papers describing the multigrid method as we know it
now~\cite[Section~10.9.2]{HackbuschBookIter}, is devoted to the solution of Poisson
equations arising in time integration of 2D incompressible hydrodynamics
equations~\cite{Fedorenko_hist}.
Currently, multigrid methods form a major tool for efficient implementation
of implicit and semi-implicit time integration schemes
on parallel
supercomputers~\cite{AmaladasKamath1999,Gerlinger_ea2001,ZhukovNovFeod2015,ZhukovFeodor2021}.

Multigrid time integration ideas have been known at least since the middle
eighties~\cite{Hackbusch84,LubO87,JanssenVandewalle1996,JanssenVandewalle1997}.
The approach proposed in these works is essentially based on the
waveform relaxation methods~\cite{White_ea1985,Vandewalle1993}, also known as
dynamic iteration methods~\cite{MiekkalaNevanlinna1987}.
The method we propose here is different in the sense that it is designed specifically for
iterative methods evaluating the matrix exponential and $\varphi$ function
and does not employ the waveform relaxation framework.
Thus, implementation issues typical for the waveform relaxation methods,
such as storing approximate solutions across time efficiently and accurately,
do not have to be addressed.  This allows to keep our approach rather
simple.
However, one essential similarity of our approach to that of~\cite{Hackbusch84,LubO87}
is that the residual concept is crucial in both settings.
Note that the multigrid methods have been playing a key role in the recent revival
of time-parallel methods, see, e.g.,~\cite{Falgout_ea2014,Minion_ea2015,Gander2015}.

This paper is organized as follows.
\added{In the remaining part of this section a brief introduction to the basic
ideas of the multigrid method is given.}
The next section is devoted to the problem
setting and some preliminaries concerning iterative \deleted{Krylov subspace} evaluation
of the $\varphi$ function. 
To be specific in our presentation, we consider only the $\varphi$ matrix function
evaluations.
This covers the case of the matrix exponential, as the key relation 
being evaluated reduces
to a matrix exponential action for the source vector set to zero
(namely, formula~\eqref{yt} with $g=0$).
In Section~\ref{s:CGC}, our coarse grid correction algorithm is presented
and analyzed, first its two-grid and then its multigrid versions.
Numerical experiments and their results are discussed in Section~\ref{s:num_exp}.
The last section contains some conclusions and an outlook to
further research.

\subsection{\added{Basic multigrid concepts}}
Since the approach presented here is essentially based on an analogy with
multigrid methods for solving linear systems, we start with briefly discussing
how a simple multigrid version, called a two-grid 
method
(\cite[Chapter~10.2]{HackbuschBookIter}, \cite[Chapter~2.2]{Trottenberg_ea2001}),
is employed to solve
a linear system
\begin{equation}
\label{ls}
A x = b,  
\end{equation}
with a nonsingular $A\in\RrNN$ and $b\in\Rr^N$ given.
We assume here that the system stems from a PDE discretization on a certain
grid and that a coarser grid discretization is available with 
a nonsingular $\tilde{A}\in\Rr^{n\times n}$, $n<N$.
By $Q\in\Rr^{N\times n}$ we denote a matrix which interpolates an approximate
solution $\tilde{x}\in\Rr^n$ defined on the coarse grid to an approximate
solution $Q\tilde{x}\in\Rr^N$ on the fine grid.  Then $Q^T$ can be seen
as an opposite operation, which restricts a fine grid solution to the coarse
grid.
In multigrid methods $Q$ and $Q^T$ are called respectively prolongation and
restriction operators between the two grids.
In general, these operators do not have to be linear
(in which case their action, of course, cannot be realized
by a matrix-vector multiplication).
For simplicity of presentation prolongation and restriction are
assumed to be linear in this section.

Let $A$ be such that a conventional iterative scheme
\begin{equation}
\label{it}
x_{\text{next}} = M^{-1}(M-A)x_{\text{current}} + M^{-1}b,
\end{equation}
where $M\in\RrNN$ is a nonsingular matrix, converges to the solution of~\eqref{ls}.
Here $M$ represents a part of $A$ such that systems with $M$ can be solved easily
(for instance, $M$ being the diagonal part of $A$ yields the Jacobi iteration).
The multigrid method is based on the key observation that iterative schemes~\eqref{it}
typically have the so-called smoothing properties: the ``high-frequency'' components 
of the residual $r=b-Au$, $u\approx x$, decay much faster than the ``low-frequency'' components.
For symmetric positive definite $A$ (usually being a discretization of an elliptic partial
differential operator) the low-frequency components are often defined as the eigenvector
contributions corresponding to the half smallest in magnitude eigenvalues,
and the high-frequency components correspond to the other half. 
For a more detailed discussion 
see, e.g., \cite[relation~(9.5)]{Fedorenko1973}, \cite[Chapter~2.1]{Trottenberg_ea2001},
or \cite[Section~10.1.1]{HackbuschBookIter}.
In Figure~\ref{f:smooth} we present an illustration of this smoothing effect.
As the bottom right plot in Figure~\ref{f:smooth} suggests, a good smoothing
can also be obtained by applying the restriction followed by prolongation
operations.  This will be essentially used in our coarse grid correction
approach presented below.

\begin{figure}
\begin{center}
\begin{tabular}{cc}
\includegraphics[width=0.45\linewidth]{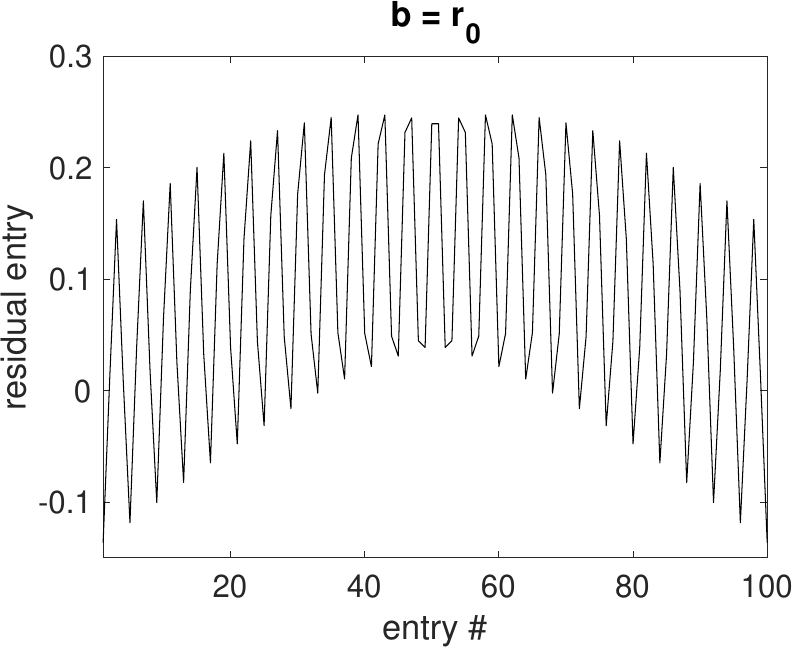} &
\includegraphics[width=0.45\linewidth]{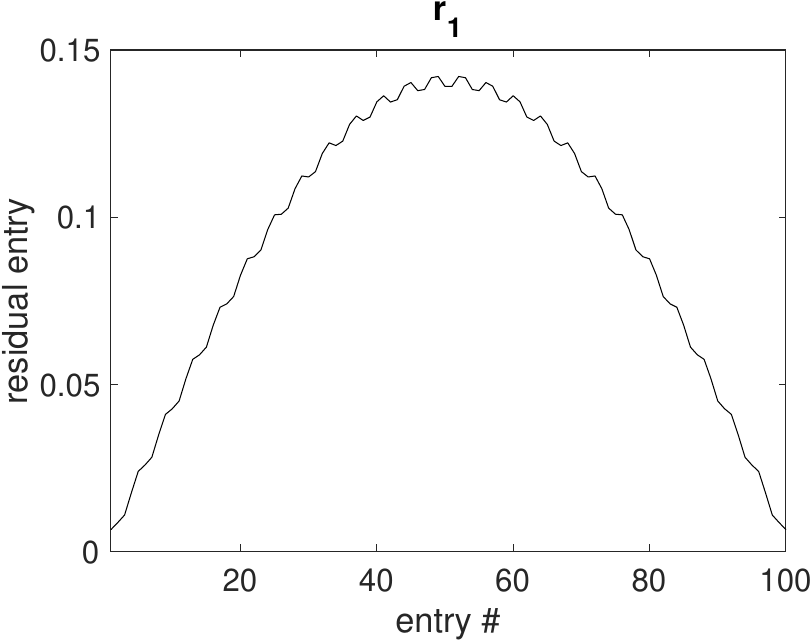} \\[1.5ex]
\includegraphics[width=0.45\linewidth]{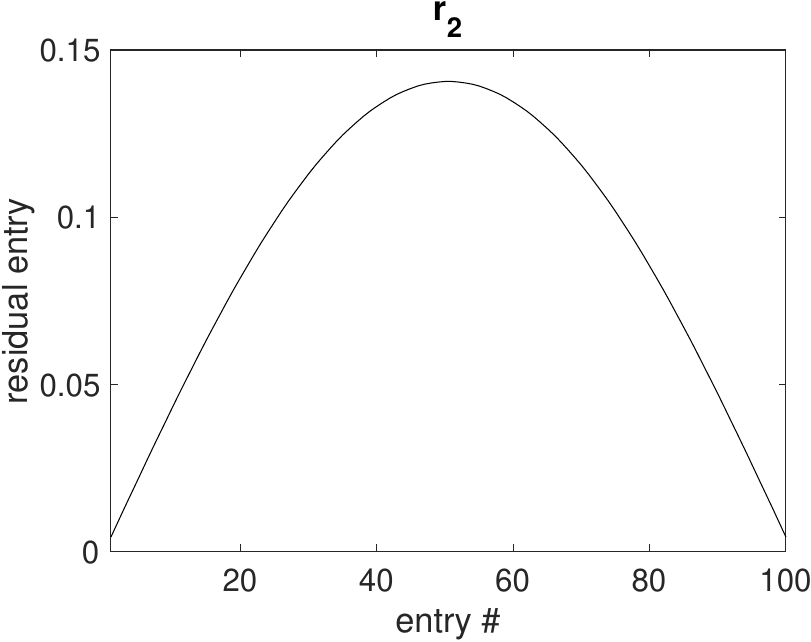} &
\includegraphics[width=0.45\linewidth]{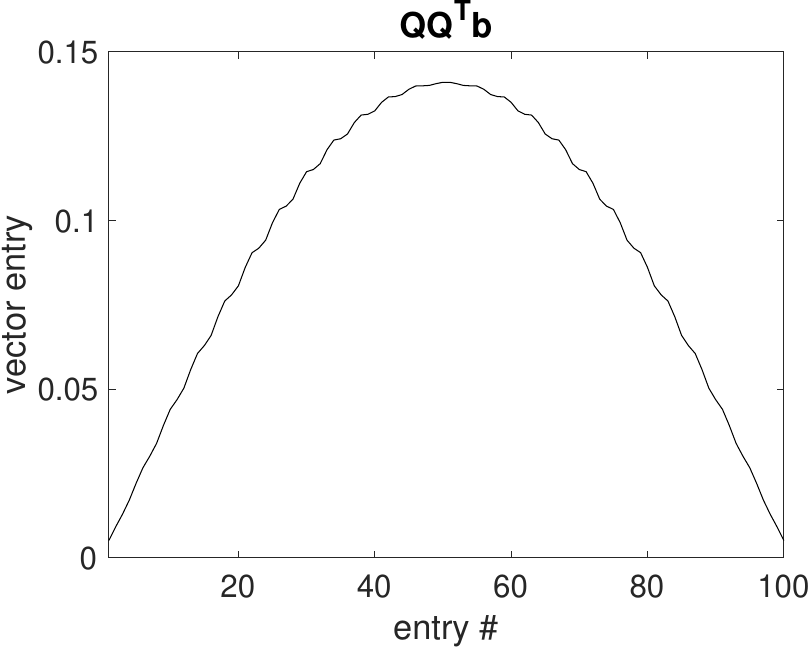}
\end{tabular}
\end{center}
\caption{\added{The smoothing properties of conventional iterative schemes
for solving linear systems $A x = b$.
$A$ is a 1D discretized Laplacian with homogeneous Dirichlet boundary conditions in the
space domain $[0,1]$. Standard second order three-point finite differences
are used on a uniform grid of $N=100$ nodes.
Zero initial guess vector $x_0=0$ is taken, so that $r_0=b$.
Top left: entries of the right-hand side vector $b$ being a sum of two
normalized eigenvectors of $A$, corresponding to the smallest in magnitude
and 51st smallest in magnitude eigenvalues.  Top right and bottom left: entries
of the residual vector at iterations~1 and~2.  Bottom right: restriction
$Q^T$
followed by prolongation $Q$ as a smoother, with a uniform coarse grid of $n=50$ nodes.}}
\label{f:smooth}
\end{figure}

If $x_m\in\Rr^N$ is an approximate solution to system~\eqref{ls}, an iteration
update $x_m\rightarrow x_{m+1}$ in the two-grid method can be carried out as 
shown in Figure~\ref{f:2G}.  The key idea here is that, since the residual $\bar{r}_m$
of $\bar{x}_m$ is smoothed at Step~1, it can be well represented on the coarse mesh
by $\tilde{r}_m$.
Hence, the corresponding correction vector $A^{-1}\bar{r}_m$
can hopefully be well approximated by the interpolated
coarse grid correction $Q\tilde{z}_m=Q\tilde{A}^{-1}\tilde{r}_m$ (Step~2).
Replacing $A^{-1}\bar{r}_m$ by $Q\tilde{A}^{-1}\tilde{r}_m$ is not only
computationally cheaper but also, if the coarse grid solution is done
iteratively, has a potential to efficiently eliminate the lower frequency modes
in the residual.
Indeed, these modes become higher frequency modes on the coarse grid and, hence,
may get within the reach of the smoothing effect.  
Note that, within the algorithmic construction in Figure~\ref{f:2G},
the smoothing steps~1 and~3 are essential because 
the coarse grid correction alone will not lead to a
converging iteration, see, e.g., \cite[Section~10.1.5]{HackbuschBookIter}
or \cite[Section~2.2.3]{Trottenberg_ea2001}.

\begin{figure}
\centerline{\begin{tabular}{rl}
\multicolumn{2}{l}{$x_{m+1} := $ \texttt{iteration2G} ($x_m$, $A\in\RrNN$, 
$\tilde{A}\in\Rr^{n\times n}$, $b$)}\\
\multicolumn{2}{l}{For given linear system $Ax=b$, $A\in\RrNN$, 
$\tilde{A}\in\Rrnn$ (a coarse grid analogue of $A$),}\\
\multicolumn{2}{l}{$x_m\in\Rr^N$, 
carries out an iteration $x_m\rightarrow x_{m+1}$ of the two-grid method}
\\[1.5ex]
1.  & Smoothing: starting with $x_m$, carry out a number of iterations~\eqref{it}.\\
    & Store the result as $\bar{x}_m$.
\\[1.5ex]
2.  & Coarse grid correction:\\
    & Restrict the residual $\bar{r}_m = b - A\bar{x}_m$ to the coarse grid,
      $\tilde{r}_m:=Q^T\bar{r}_m\in\Rr^n$.\\
    & Find the coarse grid correction $\tilde{z}_m$ by solving $\tilde{A}\tilde{z}_m=\tilde{r}_m$.\\
    & Prolong the correction to the fine grid $z_m := Q \tilde{z}_m$.
\\[1.5ex]
3.  & Smoothing: starting with $x_m+z_m$, carry out a number of iterations~\eqref{it}.\\
    & Store the result as $x_{m+1}$.
\end{tabular}}
\caption{A two-grid method iteration for solving linear system $Ax=b$}
\label{f:2G}  
\end{figure}

\section{Problem setting \added{and baseline methods}}
\added{Having discussed some basic concepts of multigrid for solving
linear systems, we are now ready to turn to time-dependent problems
and related matrix functions.}
Unless reported otherwise, \replaced{in this}{through out the} paper
\added{$(\,\cdot\,,\,\cdot\,)$ denotes the Euclidean
inner product and}
$\|\cdot\|$ denotes the Euclidean vector or the corresponding operator norm.
For given $A\in\RrNN$, $v,g\in\Rr^N$, and $T>0$ we are interested in solving
initial-value problem (IVP)
\begin{equation}
\label{ivp}
y'(t) = - Ay(t)+g, \quad y(0)=v, \quad t\in[0,T].
\end{equation}
Through out this paper 
we assume that the symmetric part $\frac12(A+A^T)$
of $A$ is a positive semidefinite matrix, i.e.,
there exists a constant $\omega\geqs 0$ such that
\begin{equation}
\label{omega}
\dfrac{(\frac12(A+A^T)x,x)}{(x,x)}\geqs\omega\geqs 0, \quad
\text{for all}\quad x\in\Rr^N, x\ne 0.
\end{equation}
We also assume that relation~\eqref{omega}, with a different constant
$\tilde{\omega}\added{\geqs}0$, 
holds for \replaced{the }{a} coarse grid analogue \added{$\tilde{A}$} of $A$. 
\deleted{denoted by $\tilde{A}\in\Rr^{n\times n}$, $n<N$.}
Furthermore, we assume that $v$ and $g$ are not simultaneously zero.
It is easy to check that
\begin{equation}
\label{yt}
y(t) = v + t\varphi(-tA)(g-Av),\quad t\geqs 0,
\end{equation}
where $\varphi(-tA)$ is a matrix function with $\varphi$ defined as
\begin{equation}
\label{phi}
\varphi(z)\equiv \frac{e^z-1}{z}, \quad z\in\Cc\setminus\{0\},
\qquad \varphi(0) \equiv 1. 
\end{equation}
Note that for $g=0$ relation~\eqref{yt} takes a form
\begin{equation}
\tag{\ref{yt}$'$}
y(t) = \exp(-tA)v,\quad t\geqs 0,
\end{equation}
where $\exp(A)$ is the matrix exponential.

\subsection{\added{Krylov subspace methods}}
\label{s:Krylov}
Assume we solve IVP~\eqref{ivp} by computing $y(t)$ in~\eqref{yt} 
by the regular (polynomial) Krylov subspace method
(see, e.g.,~\cite{GolVanL,SaadBook2003,Henk:book})
and obtain, after $k$ steps 
of the Arnoldi (or Lanczos) process, an approximate Krylov subspace solution $y_k(t)$.
This means that matrices $V_{k+1}\in\Rr^{N\times(k+1)}$ and
$H_{k+1,k}\in\Rr^{(k+1)\times k}$ are constructed such that the columns
$v_1$, \dots, $v_{k+1}$ of $V_{k+1}$ are orthonormal and span the Krylov subspace,
with
\begin{equation}
\label{beta_v1}
v_1 = \frac1{\beta}\bar{g}, \quad \beta=\|\bar{g}\|, \quad \bar{g}=g-Av.  
\end{equation}
The matrix $H_{k+1,k}$ is upper Hessenberg and it holds
\begin{equation}
\label{Arn}
AV_k = V_{k+1} H_{k+1,k},
\end{equation}
where the right hand side can be rewritten as
$V_{k+1} H_{k+1,k}= V_kH_{k,k} + h_{k+1,k}v_{k+1}e_k^T$, with $H_{k,k}$ being the leading principal
$k\times k$ submatrix of $H_{k+1,k}$, $h_{k+1,k}$ being the $(k+1,k)$ entry of $H_{k+1,k}$
and $e_k=(0,\dots,0,1)^T\in\Rr^k$.
Relation~\eqref{Arn} is usually refered to as Arnoldi decomposition.
The Krylov subspace solution $y_k(t)$, solving~\eqref{ivp} approximately,
then reads
\begin{equation}
\label{yk}  
y_k(t) = v + V_ku(t), \quad u(t)= t\varphi(-tH_{k,k})\beta e_1.
\end{equation}
Approximation quality of $y_k(t)$ can be determined by measuring
the exponential residual~\cite{CelledoniMoret97,DruskinGreenbaumKnizhnerman98,BGH13}
$$
r_k(t)\equiv -Ay_k(t)-y_k'(t)+g,
$$
which is readily available in the course of the Krylov subspace
iterations.  Indeed, it is easy to check that~\cite{BKT21}
\begin{equation}
\label{rk}
r_k(t)= -h_{k+1,k}v_{k+1}e_k^Tu(t), 
\end{equation}
where $u(t)$ is introduced in~\eqref{yk}.

If it is impossible or inefficient to store and handle more than $k+1$
Krylov subspace vectors, one can restart the Krylov subspace method
after $k$ steps~\cite{EiermannErnst06,TalEzer2007,Afanasjew_ea08,PhD_Guettel,%
Eiermann_ea2011}.
This can be done, for instance, as follows.
Denoting by $\yex(t)$ the exact solution of~\eqref{ivp}, we see that
the error $\eps_k(t)=\yex(t)-y_k(t)$ satisfies
\begin{equation}
\label{ivp0}
\eps_k'(t) = - A\eps_k(t)+r_k(t), \quad \eps_k(0)=0, \quad t\in[0,T].
\end{equation}
Solving~\eqref{ivp0} with another $k$ Krylov subspace iterations
we obtain its approximate solution $\tilde{\eps}_k(t)\approx\eps_k(t)$
and update
\begin{equation}
\label{upd}
y_{2k}(t)=y_k(t) + \tilde{\eps}_k(t). 
\end{equation}
This \emph{residual restarting} procedure is proposed and used 
in~\cite{CelledoniMoret97,DruskinGreenbaumKnizhnerman98,BGH13}.
It is not difficult to check that the residual of $y_{2k}(t)$ is then
again a scalar time-dependent function times a constant vector.

Another restarting procedure is based on the observation
that $\|r_k(t)\|$ is a monotonically increasing function of $t$.
Hence, for any tolerance $\tol$ it is possible to find a $\delta>0$
such that $\|r_k(s)\|\leqs\tol$ for all $s\in[0,\delta]$.  We
can then compute $y_k(\delta)$ and restart by setting
in~\eqref{ivp} $v:=y_k(\delta)$ and shortening the time interval
$t:=t-\delta$.  This is called
residual-time (RT) restarting~\cite{ART,BKT21}.
For other restarting techniques we refer 
to~\cite{FrommerGuettelSchweitzer2014a,FrommerGuettelSchweitzer2014b,JaweckiAuzingerKoch2018}.

From~\eqref{ivp0} we see that the residual $r_k(t)$ can be regarded
as a backward error of the approximate solution $y_k(t)$.
IVP~\eqref{ivp0} also allows
to obtain the following error estimate.

\begin{lemma}
\emph{\cite{BGH13}}
Let $A\in\Rr^{N\times N}$ satisfy~\eqref{omega} and let $\yex(t)$ be
the exact solution of~\eqref{ivp}.  If $r_k(t)$ is the residual
of an approximate solution $y_k(t)\approx\yex(t)$ then
the error $\eps_k(t)=\yex(t)-y_k(t)$ can bounded in norm for
any $t\geqs 0$ as
\begin{equation}
\label{err2res}
\|\eps_k(t)\|\leqs t\varphi(-t\omega)\max_{s\in[0,t]}\|r_k(s)\|.
\end{equation}
\end{lemma}
Note that for any $\omega\geqs 0$ and any $t\geqs 0$ we have
\begin{equation}
\label{phi_bound}
t\varphi(-t\omega) =
\begin{cases}
  \frac{1 - e^{-t\omega}}{\omega}\leqs \min\{ t, \frac1\omega \}
                               \leqs t, &\quad \text{for}\; \omega>0,
  \\
  t, &\quad \text{for}\; \omega=0.
\end{cases}
\end{equation}

\subsection{\added{Chebyshev polynomial expansion}}
Another important class of iterative methods for computing actions of
the matrix exponential and related functions are methods based on
Chebyshev polynomial expansion (see, e.g.,~\cite{TalEzer89},
\cite[Section~3.2.3]{RyabenkiiTsynkov}).
These methods are usually applied for symmetric
and skew-symmetric matrices.  
If the matrix $tA$ is transformed in such a way
that its eigenvalues lie in the interval $[-1, 1]$, the 
Chebyshev polynomial expansion reads
\begin{equation}
\label{Cheb}
y(t)=\exp(-tA)v\approx y_k(t) = P_k(-tA)v = 
\left[ \sum_{j=1}^k c_jT_j(-tA) +\frac{c_0}2I\right]v.   
\end{equation}
This expansion can be computed recursively by the Clenshaw algorithm~\cite{Clen62}.
In~\cite{BGH13} we have modified the algorithm in such a way that recursions
for the residual $r_k(t)=-Ay_k(t)-y'_k(t)$ are carried out as well.
Then, the iterative process can be stopped as soon as the residual norm
is small enough, see~\cite[Section~3]{BGH13} for details.
Just as the Krylov subspace iterations,
this iterative procedure with a residual-based stopping criterion
can be readily applied in combination with our coarse grid 
correction approach which we present in the next section.
Note that since our Chebyshev algorithm computes the action of the 
matrix exponential, it can be employed to solve~\eqref{ivp}
for $g=0$, cf.~(\ref{yt}$'$).  

\section{Coarse grid corrections}
\label{s:CGC}
\subsection{Coarse grid corrections, a two-grid version}
We now describe our coarse grid correction (CGC) algorithm.
We assume that $A$ stems from a PDE operator discretization
on a certain grid and that a coarser grid exists with $\tilde{A}$
being the coarse grid counterpart of $A$.  
\added{Recall that $Q$ and $Q^T$ are the matrices of respectively prolongation and
restriction operators between the two grids, assumed to be linear
for simplicity of presentation.}
The sought after solution $y(t)$ can be computed as an action of the
$\varphi$ matrix function according to formula~\eqref{yt}, i.e.,
$$
y(t) := v + t\varphi(-tA)\bar{g}, 
$$
where we denote $\bar{g}=g-Av$.
Our approach is based on splitting the vector $\bar{g}$ into a part
which can be well represented on the coarse grid, namely $QQ^T\bar{g}$,
and the remaining part $\hat{g}=\bar{g} - QQ^T\bar{g}$.
The $\varphi$ matrix function is then evaluated separately on
$QQ^T\bar{g}$ and on $\hat{g}$. 
Since $QQ^T\bar{g}$ is a smooth vector we can hope that
the matrix function action $t\varphi(-tA)QQ^T\bar{g}$ can be replaced
by its extrapolated coarse grid analogue $tQ\varphi(-t\tilde{A})Q^T\bar{g}$.
The remaining non-smooth component $\hat{g}$ is then
handled on the original fine grid.  If $\|\hat{g}\|/\|\bar{g}\|$ is small
then the action $t\varphi(-tA)\hat{g}$ can be evaluated with a relaxed
tolerance.

An algorithmic description of our CGC algorithm
is presented in Figure~\ref{f:alg_2g}.
There, the algorithms computing the $\varphi$ action at steps~1 and~2
are supposed to produce approximate solutions $\tilde{y}(t)$ and $\hat{y}(t)$
such that their residuals
\begin{equation}
\label{resids}
\tilde{r}(t)\equiv -\tilde{A}\tilde{y}(t)-\tilde{y}'(t)+\tilde{g},
\quad
\hat{r}(t)\equiv -A\hat{y}(t)-\hat{y}'(t)+\hat{g}  
\end{equation}
satisfy, respectively, 
\begin{equation}
\label{stop}
\max_{s\in[0,t]}\|\tilde{r}(s)\|\leqs\tilde{\beta}\,\widetilde{\tol},
\quad
\max_{s\in[0,t]}\|\hat{r}(s)\|\leqs\hat{\beta}\,\widehat{\tol}.
\end{equation}

\begin{figure}
\centerline{\begin{tabular}{rl}
  \multicolumn{2}{l}{$\ymg(t) := $ \texttt{CGC2G} ($A$, $v$, $g$, $t$, $\tol$)}\\
  \multicolumn{2}{l}{For given $A\in\Rr^{N\times N}$, $v,g\in\Rr^N$,
    $t>0$ and tolerance $\tol$, the algorithm computes}\\
  \multicolumn{2}{l}{$\ymg(t):\approx v + t\varphi(-t A)(g-Av)$ whose
    error is bounded as shown in Proposition~\ref{err_est}.}
  \\[1.5ex]
  0. & Form a coarse grid analogue $\tilde{A}\in\Rr^{n\times n}$ of $A$,
       set $\bar{g}=g-Av$.\\
     & Split $\bar{g}$: $\tilde{g}:=Q^T\bar{g}$, $\hat{g} := \bar{g}-QQ^T\bar{g}$,
    $\beta:=\|\bar{g}\|$, $\tilde{\beta}:=\|\tilde{g}\|$,
    $\hat{\beta}:=\|\hat{g}\|$.
    \\[1.5ex]
  1. & Compute 
  $\tilde{y}(t):\approx t\varphi(-t\tilde{A})\tilde{g}$ with tolerance
  $\widetilde{\tol}:=(\beta/\tilde{\beta})\,\tol$ (coarse grid).
  \\[1.5ex]
  2. & Compute 
  $\hat{y}(t):\approx t\varphi(-t A)\hat{g}$ with tolerance
  $\widehat{\tol}:=(\beta/\hat{\beta})\,\tol$ (fine grid).
  \\[1.5ex]
  3. & Form the sought after approximate solution
       $\ymg(t) := v + \hat{y}(t) + Q\tilde{y}(t)$.
\end{tabular}}
\caption{A two-grid version of our CGC (coarse grid correction) algorithm}
\label{f:alg_2g}
\end{figure}

A proposition below reveals the structure of the error
of the CGC algorithm.  As we will see, the error contains a term
which can not be made arbitrarily small by using a stringent tolerance.
The accuracy of the method is restricted and, as expected, depends on how well
the smooth part of the solution can be approximated by the extrapolated
coarse grid solution. 

\begin{proposition}
\label{err_est}
Let $A\in\Rr^{N\times N}$ and its coarse grid analogue $\tilde{A}\in\Rr^{n\times n}$
satisfy relation~\eqref{omega} and let $\yex(t)$ be the exact solution of~\eqref{ivp}.
Then for solution $\ymg(t)$ of the two-grid CGC algorithm
(see Figure~\ref{f:alg_2g}) holds, for any $t\geqs 0$,
\begin{align}
\label{err_est1}
\|\yex(t)-\ymg(t)\| &\leqs
t\left\|(\varphi(-tA)Q-Q\varphi(-t\tilde{A}))\tilde{g}\right\| +
t\varphi(-t\bar{\omega})(\|Q\|+1)\beta\,\tol
\\
\label{err_est2}
&\leqs
\left(t\varphi(-t\bar{\omega})\right)^2
\|Q\tilde{A} - A Q\|\, \|\tilde{g}\| +
t\varphi(-t\bar{\omega})(\|Q\|+1)\beta\,\tol,
\end{align}
where $\bar{\omega}=\min\{\omega,\tilde{\omega}\}$.
\end{proposition}

\begin{proof}
For simplicity assume, without loss of generality, that $v=0$ in~\eqref{ivp}
and $g\ne 0$.
Let $S_t\equiv t\varphi(-t A)$ and $\tilde{S}_t\equiv t\varphi(-t\tilde{A})$, where
the subscripts indicate that the matrices $S_t$ and $\tilde{S}_t$
depend on the parameter $t$.
We have $g = QQ^Tg + (I-QQ^T)g = Q\tilde{g} + \hat{g}$ and, hence, 
we can split the exact solution $\yex(t)$ of~\eqref{ivp} as 
$$
\yex(t) = S_t g = S_t Q\tilde{g} + S_t\hat{g}.
$$
Since an approximate solution is the exact solution of a problem perturbed by
its residual, we can write 
$$
\ymg(t) = Q\tilde{y}(t) + \hat{y}(t) = Q\tilde{S}_t(\tilde{g}-\tilde{r}(t))
+ S_t(\hat{g}-\hat{r}(t)),
$$
where $\tilde{r}(t)$ and $\hat{r}(t)$ are the residuals of the
approximate solutions 
$\tilde{y}(t)$ and $\hat{y}(t)$, respectively, see~\eqref{resids}.
We then can estimate
\begin{equation}
\label{err_est1a}
\begin{aligned}
 \|\yex(t)-\ymg(t)\| &=\|S_t Q\tilde{g} + S_t\hat{g} - Q\tilde{S}_t(\tilde{g}-\tilde{r}(t)) -S_t(\hat{g}-\hat{r}(t))\|
  \\
&= 
  \|S_t Q\tilde{g} - Q\tilde{S}_t\tilde{g} + Q\tilde{S}_t\tilde{r}(t) + S_t\hat{r}(t)\|
  \\
&\leqs 
  \|S_t Q\tilde{g} - Q\tilde{S}_t\tilde{g}\| + \|Q\tilde{S}_t\tilde{r}(t)\| + \|S_t\hat{r}(t)\|
\\
&\leqs
\|S_t Q\tilde{g} - Q\tilde{S}_t\tilde{g}\| + 
t\varphi(-t\tilde{\omega})\,\|Q\|\max_{s\in[0,t]}\|\tilde{r}(s)\|
\\
&
\phantom{\leqs\|S_t Q\tilde{g} - Q\tilde{S}_t\tilde{g}\|} +
t\varphi(-t\omega)\max_{s\in[0,t]}\|\hat{r}(s)\|
\\
&\leqs
\|S_t Q\tilde{g} - Q\tilde{S}_t\tilde{g}\| + 
t\varphi(-t\bar{\omega})\,(\|Q\|\tilde{\beta}\,\widetilde{\tol} + \hat{\beta}\,\widehat{\tol})
\\
&=
\|S_t Q\tilde{g} - Q\tilde{S}_t\tilde{g}\| + 
t\varphi(-t\bar{\omega})\,(\|Q\| + 1)\beta\,\tol,
\end{aligned}
\end{equation}
where we take into account the residual-based error estimate~\eqref{err2res},
the stopping criteria~\eqref{stop} and the choice of the tolerances $\widetilde{\tol}$ and
$\widehat{\tol}$ in the two-grid CGC algorithm (see Figure~\ref{f:alg_2g}).
The last inequality is the error estimate~\eqref{err_est1}, the first one
of the two to be proved. 

Note that $\tilde{S}_t\tilde{g}$, appearing in the obtained estimate,
is the exact solution of the coarse grid IVP
$\tilde{y}'(t)=-\tilde{A}\tilde{y}(t) + \tilde{g}$, $\tilde{y}(0)=0$,
and denote $\ytex(t)\equiv\tilde{S}_t\tilde{g}$.
Then the prolonged exact coarse grid solution $Q\tilde{S}_t\tilde{g}=Q\ytex(t)$ can be seen
as the exact solution $S_t(Q\tilde{g}-\rcg(t))$ of the perturbed IVP
$y'(t)=-Ay(t)+ Q\tilde{g}-\rcg(t)$, $y(0)=0$, with $\rcg(t)$ being
the residual of $Q\ytex(t)$ with respect to IVP
$$
y'(t)=-Ay(t)+ Q\tilde{g}, \quad y(0)=0.
$$
\added{Since $\ytex(t)$, by definition, satisfies $\ytex'(t)=-\tilde{A}\ytex(t)+\tilde{g}$,} 
we have for $t\geqs 0$
$$
\begin{aligned}
  \rcg(t) & \equiv -A Q\ytex(t) - Q\ytex'(t) + Q\tilde{g} =
  -A Q\ytex(t) - Q(-\tilde{A}\ytex(t)+\tilde{g}) + Q\tilde{g} =
  \\
  & = -A Q\ytex(t) + Q\tilde{A}\ytex(t) = (Q\tilde{A} - A Q)\ytex(t)
  = (Q\tilde{A} - A Q)\tilde{S}_t\tilde{g},
\end{aligned}
$$
so that
\begin{equation}
\label{est_last}  
\begin{aligned}
\|S_t Q\tilde{g} - Q\tilde{S}_t\tilde{g}\|
&=\|S_t Q\tilde{g} - S_t(Q\tilde{g}-\rcg(t))\|
=\| S_t \rcg(t) \| 
\\
&\leqs t\varphi(-t\omega)\max_{s\in[0,t]}\|\rcg(s)\|
\leqs t\varphi(-t\omega)\|Q\tilde{A} - A Q\|\max_{s\in[0,t]}\|\tilde{S}_s\tilde{g}\|
\\
&\leqs t\varphi(-t\omega)\|Q\tilde{A} - A Q\| t\varphi(-t\tilde{\omega})\|\tilde{g}\|
\\
&\leqs (t\varphi(-t\bar{\omega}))^2\|Q\tilde{A} - A Q\| \|\tilde{g}\|,
\end{aligned}
\end{equation}
where relation~\eqref{err2res} is used to bound
$\| S_t \rcg(t) \|$ and $\|\tilde{S}_s\tilde{g}\|$.
Substituting the last estimate into~\eqref{err_est1a}
we obtain~\eqref{err_est2}.
\end{proof}

\begin{note}
In exponential time integrators~\cite{HochbruckOstermann2010} the matrix
exponential and the $\varphi$ matrix function are typically evaluated
for $t:=\Delta t$, the time step size.  In this case, as estimate~\eqref{err_est2} shows,
the coarse grid error term is second order in time, i.e.,
\begin{equation}
\label{order2}  
(\Delta t\varphi(-\Delta t\bar{\omega}))^2\|Q\tilde{A} - A Q\| \|\tilde{g}\| =
O(\Delta t)^2.
\end{equation}
This means that the CGC method can be attractive within the exponential
time integration framework.
\end{note}

Since the accuracy of the CGC algorithm is restricted, it is important
to be able to estimate the achievable accuracy in practice.
Based on Proposition~\ref{err_est} and relation~\eqref{est_last},
we can estimate the coarse grid error term as
\begin{equation}
\label{pract_est}
\begin{aligned}
t\|\varphi(-tA)Q\tilde{g}-Q\varphi(-t\tilde{A})\tilde{g}\|   
&\leqs t\varphi(-t\omega)\max_{s\in[0,t]}\|\rcg(s)\|
\\
&\approx t\varphi(-t\omega)\|\rcg(t)\| =
t\varphi(-t\omega)\|(Q\tilde{A}-AQ)\ytex(t)\|\\
&\approx t\varphi(-t\omega)\|(Q\tilde{A}-AQ)\tilde{y}(t)\|,
\end{aligned}
\end{equation}
which is an easily computable estimate. 
It is convenient to compute the estimate after
step~1 of the Algorithm (see Figure~\ref{f:alg_2g}),
as soon as $\tilde{y}(t)$ becomes available.
To estimate the value of $\omega$ one can use, if within
the Krylov subspace methods framework, the Ritz values
(i.e., the eigenvalues of the projected matrix $H_{k,k}$).

\subsection{CGC algorithm, a multigrid version}
If the grid size is large, to solve the coarse grid problem
$\tilde{y}(t):\approx t\varphi(-t\tilde{A})\tilde{g}$ at step~1,
we can again apply the coarse grid correction.  This results in
a recursive multigrid algorithm presented in Figure~\ref{f:alg_mg}.
The algorithm differs from the two-grid algorithm in Figure~\ref{f:alg_2g}
only in step~1.
Assume that the algorithm uses a sequence of $m$ grids numbered
such that grid~1 is the finest and grid~$m$ is the coarsest one.
If $Q_j$ is the linear prolongation operator from \replaced{grid}{mesh}~$j+1$ to grid~$j$
and $Q_j^T$ is the corresponding restriction operator
then the input vector $\tilde{g}_1:=\bar{g}$ is successively split as
\begin{equation}
\label{split_mg}
\begin{alignedat}{3}
  &\text{grid 1:} \quad &\tilde{g}_1 &= Q_1\tilde{g}_2 + \hat{g}_1, \quad&\text{with}\quad&
  \tilde{g}_2:=Q_1^T\bar{g},\;\hat{g}_1:=\bar{g}-Q_1\tilde{g}_2,
  \\
  &\text{grid 2:} \quad &\tilde{g}_2 &= Q_2\tilde{g}_3 + \hat{g}_2, \quad&\text{with}\quad&
  \tilde{g}_3:=Q_2^T\tilde{g}_2,\;\hat{g}_2:=\tilde{g}_2-Q_2\tilde{g}_3,
  \\
  &\dots  &  &  &  &
  \\
  &\text{grid $m-1$:} \quad &\tilde{g}_{m-1} &= Q_{m-1}\tilde{g}_m + \hat{g}_{m-1},
  \quad& \text{with}\quad& \tilde{g}_m:=Q_{m-1}^T\tilde{g}_{m-1},
  \\
  & & & & &\hat{g}_{m-1}:=\tilde{g}_{m-1}-Q_{m-1}\tilde{g}_m.  
\end{alignedat}
\end{equation}
Note that the $\varphi$ matrix functions are evaluated once at step~1 of the algorithm
on the coarsest grid~$m$ and $m-1$ times at step~2 on grids $1$, \dots, $m-1$.

\begin{figure}
  \centerline{\begin{tabular}{rl}
  \multicolumn{2}{l}{$\ymg(t) := $ \texttt{CGCMG} ($A$, $v$, $g$, $t$, $\tol$)}\\
  \multicolumn{2}{l}{For given $A\in\Rr^{N\times N}$, $v,g\in\Rr^N$,
    $t>0$ and tolerance $\tol$, the algorithm computes}\\
  \multicolumn{2}{l}{$\ymg(t):\approx v + t\varphi(-t A)(g-Av)$ whose
    error is bounded as shown in Proposition~\ref{err_est}.}
  \\[1.5ex]
  0. & Form a coarse grid analogue $\tilde{A}\in\Rr^{n\times n}$ of $A$,
       set $\bar{g}=g-Av$.\\
     & Split $\bar{g}$: $\tilde{g}:=Q^T\bar{g}$, $\hat{g} := \bar{g}-QQ^T\bar{g}$,
    $\beta:=\|\bar{g}\|$, $\tilde{\beta}:=\|\tilde{g}\|$,
    $\hat{\beta}:=\|\hat{g}\|$.
    \\[1.5ex]
  1. & If grid is coarse enough then\\
     & \hspace*{3em}compute 
       $\tilde{y}(t):\approx t\varphi(-t\tilde{A})\tilde{g}$ with tolerance
       $\widetilde{\tol}:=(\beta/\tilde{\beta})\,\tol$\\
     & else\\
     & \hspace*{3em}recursion: $\tilde{y}(t) := $
                               \texttt{CGCMG} ($\tilde{A}$, $v=0$, $\tilde{g}$,
                               $t$, $\widetilde{\tol}=(\beta/\tilde{\beta})\,\tol$).
  \\[1.5ex]
  2. & Compute 
  $\hat{y}(t):\approx t\varphi(-t A)\hat{g}$ with tolerance
  $\widehat{\tol}:=(\beta/\hat{\beta})\,\tol$ (fine grid).
  \\[1.5ex]
  3. & Form the sought after approximate solution
       $\ymg(t) := v + \hat{y}(t) + Q\tilde{y}(t)$.
\end{tabular}}
\caption{A multigrid version of the CGC algorithm}
\label{f:alg_mg}
\end{figure}

\begin{proposition}
\label{err_mg_pr}
Let the recursive multigrid CGC algorithm (see Figure~\ref{f:alg_mg})
be applied on a sequence of grids $j=1,\dots,m$ such that grid~$j+1$ is
coarser than grid~$j$ for all $j=1,\dots,m-1$.
Let $A_j\in\Rr^{n_j\times n_j}$, $j=1,\dots,m$, be discretizations of a certain PDE operator
on grid~$j$ which satisfy relation~\eqref{omega} with $\omega=\omega_j$
and let $\bar{\omega}=\min_{j=1,\dots,m}\omega_j$.
Furthermore, let $Q_j$, $j=1,\dots,m-1$, be linear prolongation operators
from \replaced{grid}{mesh}~$j+1$ to grid~$j$ and let $\yex(t)$ be the exact solution of~\eqref{ivp}.
If the tolerances in the $\varphi$ matrix function evaluations
(steps~1 and~2 of the algorithm) are chosen such that
\begin{equation}
\label{tols}
\max_{s\in[0,t]}\|\tilde{r}_m(s)\| \leqs \beta\,\tol,
\qquad
\max_{s\in[0,t]}\|\hat{r}_j(s)\| \leqs \beta\,\tol, \quad j=1,\dots,m-1,
\end{equation}
where $\beta=\|\bar{g}\|$, and $\tilde{r}_j(s)$, $\hat{r}_j(t)$ are the residuals of
the solvers employed respectively at steps~1 and~2 of the algorithm on grid~$j$, 
then for solution $\ymg(t)$ of the multigrid CGC algorithm holds, for any $t\geqs 0$,
\begin{equation}
\label{err_mg_eq}
\begin{aligned}
\|\yex(t)-\ymg(t)\| \leqs &\sum_{j=1}^{m-1}\prod_{i=1}^{j-1}\|Q_i\|\,
t \left\|\left(\varphi(-tA_j)Q_j - Q_j\varphi(-tA_{j+1})\right)\tilde{g}_{j+1}\right\|
\\
& + t \varphi(-t\bar{\omega})\beta\,\tol \sum_{j=1}^m\prod_{i=1}^{j-1}\|Q_i\| .
\end{aligned}
\end{equation}
\end{proposition}


\begin{proof}
Let $S_t^{(j)}$ denote the solution operator on grid $j$, i.e.,
$S_t^{(j)} \equiv t\varphi(-tA_j)$,
and let $G_t^{(j)} \equiv S_t^{(j)}Q_j - Q_jS_t^{(j+1)}$, $j=1,\dots,m-1$.
Furthermore, let $e_j(t)\equiv \yex^{(j)}(t)-\ymg^{(j)}(t)$ be the error triggered
by the multigrid CGC algorithm on grid~$j=1,\dots,m-1$, with
$$
\yex^{(j)}(t)\equiv S_t^{(j)}\tilde{g}_j,
$$
$\ymg^{(j)}(t)$ being the algorithm solution on grid~$j$
and $\tilde{g}_j$ is defined in~\eqref{split_mg}.
The algorithm solutions $\ymg^{(j)}(t)$ on grid~$j$,  $j=1,\dots,m-1$,
satisfy a recurrence relation
\begin{equation}
\label{y_rec}
\ymg^{(j)}(t) = Q_j\ymg^{(j+1)}(t) + S_t^{(j)}(\hat{g}_{j}-\hat{r}_j(t)),
\end{equation}
where $\hat{r}_j(t)$ is the residual of the solver at step~2 of the algorithm,
see Figure~\ref{f:alg_mg}.
Note that $e_{m-1}(t)$ can be estimated in the same way as the error of the
two-grid CGC algorithm (see proof of Proposition~\ref{err_est}).  Indeed,
$$
\begin{aligned}
  \yex^{(m-1)}(t) &= S_t^{(m-1)}Q_{m-1}\tilde{g}_m + S_t^{(m-1)}\hat{g}_{m-1},\\
  \ymg^{(m-1)}(t) &= Q_{m-1}S_t^{(m)}(\tilde{g}_m-\tilde{r}_m(t)) +
                   S_t^{(m-1)}(\hat{g}_{m-1}-\hat{r}_{m-1}(t)),
\end{aligned}
$$
where $\tilde{r}_m(t)$ and $\hat{r}_{m-1}(t)$ are respectively the residuals of the solvers
in step~1 (``then'' branch of the if statement) and step~2 of the algorithm.
Hence,
\begin{equation}
\label{emm}
\begin{aligned}  
  \|e_{m-1}(t)\| &\leqs \|S_t^{(m-1)}Q_{m-1}\tilde{g}_m - Q_{m-1}S_t^{(m)}\tilde{g}_m\| \\
  & \hspace*{20ex} + \|Q_{m-1}S_t^{(m)}\tilde{r}_m(t)\| +  \|S_t^{(m-1)}\hat{r}_{m-1}(t)\|
  \\
  &\leqs \|(S_t^{(m-1)}Q_{m-1} - Q_{m-1}S_t^{(m)})\tilde{g}_m\| \\
  & \hspace*{20ex} + t\varphi(-t\bar{\omega})(\|Q_{m-1}\|+1)\beta\,\tol.
\end{aligned}
\end{equation}
For the errors $e_j(t)$, $j=1,\dots,m-2$, we obtain, substituting
$\ymg^{(j+1)}(t)= \yex^{(j+1)}(t) - e_{j+1}(t)$ in recurrence~\eqref{y_rec},
$$
e_j(t) = S_t^{(j)}Q_j\tilde{g}_{j+1} + S_t^{(j)}\hat{g}_{j}
        - Q_j(S_t^{(j+1)}\tilde{g}_{j+1}  - e_{j+1}(t))
        - S_t^{(j)}(\hat{g}_j-\hat{r}_j(t)).
$$
Therefore
\begin{equation}
\label{est_rec}
\begin{aligned}
  \|e_j(t)\| &= \|(S_t^{(j)}Q_j - Q_jS_t^{(j+1)})\tilde{g}_{j+1} + Q_je_{j+1}(t)
  + S_t^{(j)}(\hat{g}_j\hat{r}_j(t)\|
  \\
  &\leqs \|G_t^{(j)}\tilde{g}_{j+1}\| + \|Q_j\|\,\|e_{j+1}(t)\| + \|S_t^{(j)}\hat{r}_j(t)\|
  \\
  &\leqs \|G_t^{(j)}\tilde{g}_{j+1}\| + \|Q_j\|\,\|e_{j+1}(t)\| +
  t\varphi(-t\bar{\omega})\beta\,\tol.
\end{aligned}
\end{equation}
Applying the last estimate recursively for $\|e_1(t)\|$, \dots, $\|e_{m-1}(t)\|$
and using relation~\eqref{emm}, we obtain~\eqref{err_mg_eq}.
\end{proof}

Proposition~\ref{err_mg_pr} shows that, provided the prolongation operators $\|Q_j\|$
are bound\-ed in norm, the error of the multigrid CGC algorithm is, roughly
speaking, a sum of the coarse grid correction errors
$t \left\|\left(\varphi(-tA_j)Q_j - Q_j\varphi(-tA_{j+1})\right)\tilde{g}_{j+1}\right\|$,
$j=1,\dots, m-1$.
Therefore, to evaluate the accuracy of the multigrid CGC algorithm in practice
the two-grid estimate~\eqref{pract_est} can be used successively, every time
a coarse grid correction is to be carried out.
The sum of these estimates computed by~\eqref{pract_est} then can be
seen as the error estimate for the multigrid CGC algorithm.

\subsection{\added{Towards a full multigrid cycle}}
A natural question arises whether our proposed coarse grid correction approach 
can be extended to a full
V or W multigrid cycle (cf.\ a two-grid iteration in Figure~\ref{f:2G}).
Unfortunately, this appears to be more difficult than might seem at first glance.
For instance, assume that after $k$ iterations of the Krylov subspace method~\eqref{yk}
an approximate solution $y_k(t)$ is obtained along with its residual $r_k(t)$, see~\eqref{rk}.
Then, a correction to solution $y_k(t)$ 
could have been obtained by solving IVP~\eqref{ivp0} on the coarser grid,
\begin{equation}
\label{ivp1}
\begin{aligned}
\text{restriction and CGC:}&\quad 
\tilde{\eps}_k'(t) = - \tilde{A}\tilde{\eps}_k(t)+Q^Tr_k(t), \quad \tilde{\eps}_k(0)=0,
\\
\text{prolongation and update:}&\quad
y_{2k}(t)=y_k(t) + Q\tilde{\eps}_k(t). 
\end{aligned}
\end{equation}
There are two reasons why this construction turns out to be not such a good idea.
First, the Krylov subspace vectors $v_k$ typically become less and less smooth 
with growing $k$, and, recalling that $r_{k}(t)\parallel v_{k+1}$, we see that
$k$ Krylov steps are actually a bad smoother.
Second, what is even more crucial, the residual $r_{2k}(t)$ of $y_{2k}(t)$ loses 
the compact form~\eqref{rk},
which makes the whole construction hardly practical.  Indeed, to get a usable
representation for the residual $r_{2k}(t)$
a special procedure would be needed, probably based on
a sophisticated resampling and parameterizing of $y_{2k}(t)$ and $r_{2k}(t)$.

As another possible building block for extending our CGC approach to a full
multigrid cycle, Richardson waveform relaxation iteration (see, 
e.g.,~\cite[Section~5.1]{BGH13})
could be considered.  Unfortunately, the same problem of obtaining the residual
in a compact usable form arises here as well.

\subsection{\added{Evaluation of the $\varphi_k$ matrix functions}}
The $\varphi_j$ matrix functions, defined as~\cite[formula~(2.10)]{HochbruckOstermann2010}
$$
\varphi_0(z)=\exp(z),\quad
\varphi_j(z)=\int_0^1 e^{(1-\theta)z}\frac{\theta^{j-1}}{(j-1)!}\,\mathrm{d}\theta,
\quad j\geqslant 1,
$$
are instrumental in exponential time integration~\cite{HochbruckOstermann2010}.
These functions satisfy the recurrence $\varphi_{j+1}(z)=(\varphi_j(z)-\varphi_j(0))/z$,
$j\geqslant 0$, 
and it is easy to see that $\varphi_1$ defined here coincides with the $\varphi$
function defined by~\eqref{phi}.  
Our CGC algorithm can be applied to evaluate actions of $\varphi_j(-tA)$ 
using the approach of~\cite[Thm.~1]{EXPOKIT} and~\cite[Thm.~2.1]{AlmohyHigham2011}.
Indeed, assume that for a certain $j$, $1\leqslant j \leqslant p$
and given vector $g\in\Rr^N$ we have to compute $\varphi_j(-tA)g$.
This approach allows to replace computing the action of $\varphi_j(-tA)$ by computing 
the matrix exponential action of a larger augmented 
$(N+p)\times(N+p)$ matrix
$$
\widehat{A} = 
\begin{bmatrix}
A & -W \\ 0 & -J  
\end{bmatrix}, \quad
J = 
\begin{bmatrix}
0 & I_{p-1} \\ 0 & 0   
\end{bmatrix}\in\Rr^{p\times p},
$$ 
where $I_{p-1}$ is the $(p-1)\times(p-1)$ identity matrix and 
the matrix $W\in\Rr^{N\times p}$ has the vector $g$ as its first column
and is zero elsewhere.
Then we have~\cite[page~491]{AlmohyHigham2011}
\begin{equation}
\label{aug}  
\varphi_j(-tA) g = \frac{1}{t^j}\left[\exp(-t\widehat{A})e_{N+j}\right]_{1:N},
\quad 1\leqslant j \leqslant p, \quad t>0,
\end{equation}
where $[x]_{1:N}$ denotes a vector of the first $N$ entries of $x$ 
and $e_{N+j}\in\Rr^{N+p}$ is the $(N+j)$th canonical basis vector. 
As relation~\eqref{aug} shows, we can use our CGC algorithm to
accelerate computing the action $\varphi_j(-tA) g$ by applying
it to the evaluation of $\exp(-t\widehat{A})e_{N+j}$.
We note that this approach has a drawback that a possible (skew) symmetry
of $A$ is lost in the augmented matrix $\widehat{A}$ and, if this is the case,
the Lanczos process should be replaced by the more expensive Arnoldi process.
Therefore, if only actions of the $\varphi=\varphi_1$ function are required
it is advisable to evaluate $\varphi$ directly, as discussed in 
Section~\ref{s:Krylov}, rather than via~\eqref{aug}.

\section{Numerical experiments}
\label{s:num_exp}
As \deleted{a} basic iterative solver\added{s} for evaluating the $\varphi$ function on each
of the grids we take the \phiRT{} method \added{and Chebyshev polynomial solver} described 
in~\cite{BKT21} and~\cite{BGH13}, respectively.
\replaced{The first solver}{It} is a Krylov subspace method based on a polynomial Lanczos process with
a residual-based stopping criterion (see~\eqref{resids},\eqref{stop})
and the residual-time (RT) restarting procedure discussed
above~\cite{ART,BKT21}.
\added{The second solver is a modification of the Clenshaw recursion~\cite{Clen62} 
with a built in residual control~\cite{BGH13}.}
For our CGC approach it is not crucial which particular solver is employed.
Nevertheless, it is convenient to use a solver with the residual-based
stopping criterion, as this fulfills the conditions of Propositions~\ref{err_est}
and~\ref{err_mg_pr}.
 
In all the tests the Krylov subspace dimension is set to~30, which means that
the restarting takes place every 30~Krylov steps.
The errors reported for all the tests are relative error norms
$$
\frac{\|\ymg(T)-\yref(T)\|}{\|\yref(T)\|},
$$
where $\yref(T)$ is a reference solution computed by the \texttt{phiv} solver
of the EXPOKIT package~\cite{EXPOKIT}.  
Note that the reference solution is computed on the same spatial
grid, so that the relative error measured in this way displays solely
the time error.
  
All the experiments are carried out in Matlab on a Linux desktop
computer with six 2.80GHz CPUs and 16~Gb memory.
To carry out restriction and prolongation operators in all
the tests the spline interpolation is used, available
in Matlab as the \texttt{interp1} and
\texttt{interpn} functions.
Similar, less accurate results are observed if the linear interpolation
is used.

\subsection{1D heat equation}
We now present numerical experiments for one-di\-men\-sion\-al heat equation
\begin{equation}
\label{test1d}
  u_t = u_{xx} + g(x), \quad u(x,0) = 1,
  \quad g(x)=e^{-500(x-0.5)^2},\quad x\in[0,1], 
\end{equation}
where periodic boundary conditions are imposed.
The standard spatial second-order finite difference discretization of this
initial-boundary-value problem on uniform mesh $x_i=i/(N+1)$, $i=1,\dots,N$,
yields~\eqref{ivp} with $A$ being a discretized second derivative operator
$\partial^2/\partial x^2$ with periodic boundary conditions.

The stopping criterion tolerance is set to $\tol=10^{-8}$.
We take the time interval length $T=0.01$, for which
$T\|A\|_1>42\,000$ on the $N=1024$ grid and 
$T\|A\|_1>165\,000$ on the $N=2048$ grid.

The results of the test runs are presented in Table~\ref{t:1d}.
The performance of the method is evaluated in terms of the number
of matrix-vector products (matvecs), the CPU time and the reached
accuracy.
The method indicated as ``1 grid method'' is the regular \phiRT{}
Krylov subspace method run on the given grid, with no coarse grid corrections. 
The error estimates given in brackets for the 2~grid method are
computed according to~\eqref{pract_est}, the error estimates
for 3~and 4~grid methods are the sums of the estimates~\eqref{pract_est}
computed at each grid coarsening.
First, we note that the practical error estimates~\eqref{pract_est},
reported in brackets in the second table column, are by no means
sharp.  This is to be expected as the estimates are obtained by rather
crude techniques.
We see that the CGC method clearly profits from \deleted{the} splitting the
source vector $\bar{g}$ in the smooth $Q\tilde{g}$ and non-smooth
$\hat{g}$ parts.
This happens due to the relaxed tolerance values reported
under the matvec values in brackets.

Furthermore, it is instructive to compare the results of the 2~grid method for $N=1024$
and of the 3~grid method for $N=2048$.
The matvec values 25 and 1219 for the first grid should be compared to the 
corresponding matvec values 6 and~1207 for the second one. 
We see that switching to a finer spatial grid in this case hardly leads to additional
costs.  This is because the eventual over-resolution in space is compensated
by the relaxed tolerance used on the finest mesh (as the non-smooth part $\hat{g}$ 
is small in norm compared to $\bar{g}$).
Moreover, the error achieved by the 2-grid method for $N=1024$ is quite
close to the error achieved by the 3-grid method for $N=2048$.
Recall that, as discussed above, these errors reflect solely the time error and 
not the space error, which should be significantly small for the
$N=2048$ grid.
The same conclusion can be drawn by comparing the results
of the 3~grid method for $N=1024$ and of the 4~grid method for $N=2048$.

\begin{table}
\caption{Results for 1D heat equation.  The value in brackets in the second column
is the coarse grid error estimate~\eqref{pract_est}, the values in brackets under 
the matvec values are the tolerance values $\widehat{\tol}$ (the $h$ grid) and
$\widetilde{\tol}$ (coarser grids).}
\label{t:1d}
\centerline{\begin{tabular}{ccccccc}
\hline\hline
method  & error         & CPU     & \multicolumn{4}{c}{matvecs (tolerances) per grid} \\
        & (estimate)    & time, s & $h$ & $2h$ & $4h$ & $8h$
\\\hline
\multicolumn{7}{c}{grid size $N=1024$}\\
1 grid  & {\tt5.23e-14} & 1.90    & 4215 &      &   \\[1ex]
2 grid  & {\tt4.47e-08} & 0.52    & 25   & 1219 &   \\[-1ex]
      &{\fs\tt(9.9e-03)}&         & {\fs\tt(1.63e-01)}
                                         & {\fs\tt(1.41e-08)} & \\[1ex]
3 grid  & {\tt2.01e-07} & 0.37    & 25   &  444 & 409 \\[-1ex]
      &{\fs\tt(9.6e-03)}&         & {\fs\tt(1.65e-01)}
                                         & {\fs\tt(1.45e-02)}
                                                & {\fs\tt(2.00e-08)}
\\[1ex]
\hline
\multicolumn{7}{c}{grid size $N=2048$}\\
1 grid  & {\tt7.42e-14}   & 6.01     & 14508  &
\\[1ex]
2 grid & {\tt1.82e-08}    & 1.54     & 2             & 4028   \\[-1ex]
       &{\fs\tt(3.9e-02)} &          &{\fs\tt(2.64)} &{\fs\tt(1.41e-08)}
\\[1ex]
3 grid & {\tt5.97e-08}    & 0.45     & 2             & 6       & 1207 \\[-1ex]
       &{\fs\tt(3.7e-02)} &          &{\fs\tt(2.64)} &{\fs\tt(2.33e-01)} 
                                                               &{\fs\tt(2.00e-08)}
\\[1ex]
4 grid & {\tt2.12e-07}    & 0.29     & 2             & 6       &  389       &  395\\[-1ex]
       &{\fs\tt(2.6e-02)} &          &{\fs\tt(2.64)} &{\fs\tt(2.33e-01)}
                                                               &{\fs\tt(2.04e-02)}
                                                                            &{\fs\tt(2.82e-08)}
\\[1ex]
\hline
\end{tabular}}
\end{table}

\added{We now test our CGC approach combined with the Chebyshev polynomial solver.
Since the Chebyshev solver evaluates the matrix exponential rather than the
$\varphi$ matrix function, we have to change the problem setting and take
in~\eqref{test1d} $g(x)\equiv 0$, $u(x,0)=e^{-500(x-0.5)^2}$. We also set 
a smaller time interval length $T=0.001$ (for which $T\|A\|_1>16\,500$
on the $N=2048$ grid). 
In our Chebyshev polynomial solver the Chebyshev expansion is built repeatedly 
for time steps $\Delta t>0$ chosen to satisfy the requirement $\Delta t\|A\|\leqslant 1$.
The Chebyshev solver is then combined with the coarse grid correction approach
in the same way as it is done for the Krylov subspace solver, i.e., the Chebyshev solver
is employed with the residual stopping criteria as indicated in
Algorithms in Figures~\ref{f:alg_2g} and~\ref{f:alg_mg}. 
The only small adjustment made for Chebyshev iterations is
that the tolerance $\widehat{\tol}$ is relaxed to a value at most~0.1
(otherwise a moderate accuracy loss is observed).}

\added{The results for the Chebyshev polynomial solver are presented in Table~\ref{t:1d_c}.
As we see, although our CGC approach seems to work successfully,
for Chebyshev iterations it leads to a smaller efficiency gain 
than for Krylov subspace iterations.  This is not unexpected because, unlike Chebyshev iterations, 
Krylov subspace iterations adapt both to the discrete structure of the spectrum of $A$ and to 
vectors on which the matrix functions act (see, e.g.,~\cite{RateCG}).  
In Chebyshev iterations switching to a coarser 
grid leads to a gain only due to a smaller problem size and to the norm decrease
(as typically $\|\tilde{A}\|<\|A\|$ for adequate discretizations).}

\begin{table}
\caption{\added{Results for 1D heat equation with the Chebyshev polynomial
solver.  The value in brackets in the second column
is the coarse grid error estimate~\eqref{pract_est}, the values in brackets under 
the matvec values are the tolerance values $\widehat{\tol}$ (the $h$ grid) and
$\widetilde{\tol}$ (coarser grids).  Note that the results in the table
are obtained for a different test setting than those in Table~\ref{t:1d}.}}
\label{t:1d_c}
\centerline{\begin{tabular}{ccccccc}
\hline\hline
method  & error         & CPU     & \multicolumn{4}{c}{matvecs (tolerances) per grid} \\
        & (estimate)    & time, s & $h$ & $2h$ & $4h$ & $8h$
\\\hline
\multicolumn{7}{c}{grid size $N=2048$}\\
1 grid  & {\tt1.74e-08}   & 70.9     & 184\,734  &
\\[1ex]
2 grid & {\tt7.73e-05}    & 37.5     & 49\,526       & 46\,233   \\[-1ex]
       &{\fs\tt(1.4e-04)} &          &{\fs\tt(0.1)}  &{\fs\tt(1.41e-08)}
\\[1ex]
3 grid & {\tt2.31e-04}    & 32.0     & 49\,526       & 12\,340 & 11\,583 \\[-1ex]
       &{\fs\tt(5.3e-04)} &          &{\fs\tt(0.1)}  &{\fs\tt(0.1)} 
                                                               &{\fs\tt(2.00e-08)}
\\[1ex]
4 grid & {\tt5.24e-04}    & 29.0     & 49\,526       & 12\,340 &  3068      &  2915\\[-1ex]
       &{\fs\tt(1.6e-03)} &          &{\fs\tt(0.1)}  &{\fs\tt(0.1)}
                                                               &{\fs\tt(2.04e-02)}
                                                                            &{\fs\tt(2.82e-08)}
\\[1ex]
\hline
\end{tabular}}
\end{table}

\subsection{3D heat equation}
In this test we solve~\eqref{ivp} obtained by a standard 7-point second-order 
finite difference discretization of the initial-boundary-value problem
in $u=u(x,y,z,t)$
$$
\begin{gathered}    
u_t =u_{xx}+u_{yy}+u_{zz} +g(x,y,z),\quad (x,y,z)\in [0,1]^3,\\
u(x,y,z,0)=0, \quad g(x,y,z)=e^{-50(x-\frac12)^2 -100(y-\frac12)^2 -50(z-\frac12)^2},
\end{gathered}
$$
where homogeneous Dirichlet boundary conditions are imposed.
We use a uniform $n_x\times n_y\times n_z$ grid with nodes
$(x_i,y_j,z_k)$,
$$
x_i=i /(n_x+1), \quad i=1,\dots,n_x, 
$$
and $y_j$, $z_k$ defined similarly.
The grid size is taken to be $80\times 88\times 96$ and $160\times 176\times 192$,
the time interval length $T=0.1$ and the tolerance $\tol=10^{-5}$.
For these two grids we have 
$T\|A\|_1>9\,500$ and
$T\|A\|_1>37\,500$, respectively. 

\begin{table}
\caption{Results for 3D heat equation.  The value in brackets in the second column
is the coarse grid error estimate~\eqref{pract_est}, the values in brackets 
under the matvec values are the tolerance values $\widehat{\tol}$ (the $h$ grid) and
$\widetilde{\tol}$ (coarser grids).}
\label{t:3d}
\centerline{\begin{tabular}{ccccccc}
\hline\hline
method & error         & CPU     & \multicolumn{4}{c}{matvecs (tolerances) per grid} \\
       & (estimate)    & time, s & $h$ & $2h$ & $4h$ & $8h$
\\\hline
\multicolumn{7}{c}{grid size $80\times 88\times 96$}\\
1 grid & {\tt2.75e-08}    & 6.63    & 539   &        
\\[1ex]
2 grid & {\tt1.20e-03}    & 0.55    & 14    & 150   \\[-1ex]
       &{\fs\tt(9.3e-03)} &         & {\fs\tt(1.92e-01)}
                                          & {\fs\tt(2.78e-05)} 
\\[1ex]
3 grid & {\tt5.84e-03}    & 0.37    & 14    &  20 & 43 \\[-1ex]
       &{\fs\tt(2.4e-02)} &         & {\fs\tt(1.92e-01)}
                                          & {\fs\tt(2.60e-02)}
                                                & {\fs\tt(7.61e-05)}
\\[1ex]
\hline
%
%
\multicolumn{7}{c}{grid size $160\times 176\times 192$}\\
1 grid  & {\tt1.19e-09}   & 176      & 1796   &        
\\[1ex]
2 grid  & {\tt3.08e-04}   & 6.7      & 2      & 480    \\[-1ex]
        &{\fs\tt(6.7e-03)}&          &{\fs\tt(3.19)}
                                              &{\fs\tt(2.80e-05}
\\[1ex]
3 grid & {\tt1.51e-03}    & 1.48     & 2      &   5    & 146   &\\[-1ex]
       &{\fs\tt(1.6e-02)} &          &{\fs\tt(3.19)} 
                                              &{\fs\tt(5.38e-01)} 
                                                       &{\fs\tt(7.80e-05)}
\\[1ex]
4 grid & {\tt6.15e-03}    & 1.25     & 2      &   5    &   11  &  27   \\[-1ex]
       &{\fs\tt(3.2e-02)} &          &{\fs\tt(3.19)}
                                              &{\fs\tt(5.38e-01)}
                                                       &{\fs\tt(7.30e-02)}
                                                               &{\fs\tt(2.13e-04)}
\\[1ex]
\hline
\end{tabular}}
\end{table}

The results of the test runs are shown in Table~\ref{t:3d}.
As we see, the accuracy of the CGC algorithm is significantly lower in this test
problem.  This is expected because much coarser grids are used 
in this test.  Nevertheless, for moderate accuracy requirements
the achieved errors can be viewed as small enough.
The coarse grid error estimates are sharper than in the previous test,
probably due to the nonzero $\omega$ value in this test.
Again, as in the first test, the CGC algorithm 
significantly profits from the relaxed tolerance requirements due to the
smooth--non-smooth splitting of the source vector $\bar{g}$.
As a consequence, going to a finer grid leads to only a moderate
increase of computational work. 
This can be clearly seen by comparing the results obtained by the 2~grid
method on the $80\times 88\times 96$ grid and by the 3~grid method on
the $160\times 176\times 192$ grid: the number of matvec carried out by
the methods on the same grid are roughly the same (respectively,
14 and 5, 150 and 146) and the achieved errors are also similar.
Recalling again that the measured errors are the time errors
and the space error is smaller on the finer grid,  
we come to a conclusion that our CGC algorithm seems to allow
to reach a higher overall accuracy with a slightly increased work. 

\begin{table}
\caption{Results for 3D heat equation for the increased time interval length
$T=1$.  The value in brackets in the second column
is the coarse grid error estimate~\eqref{pract_est}, the values in brackets 
under the matvec values are the tolerance values $\widehat{\tol}$ (the $h$ grid) and
$\widetilde{\tol}$ (coarser grids).}
\label{t:3d_t1}
\centerline{\begin{tabular}{ccccccc}
\hline\hline
method & error         & CPU     & \multicolumn{4}{c}{matvecs (tolerances) per grid} \\
       & (estimate)    & time, s & $h$ & $2h$ & $4h$ & $8h$
\\\hline
\multicolumn{7}{c}{grid size $80\times 88\times 96$}\\
1 grid & {\tt1.27e-07}    & 8.51    & 779   &        
\\[1ex]
2 grid & {\tt1.16e-03}    & 0.55    & 14    & 150   \\[-1ex]
       &{\fs\tt(9.8e-03)} &         & {\fs\tt(1.92e-01)}
                                          & {\fs\tt(2.78e-05)} 
\\[1ex]
3 grid & {\tt5.64e-03}    & 0.37    & 14    &  20 & 53 \\[-1ex]
       &{\fs\tt(2.5e-02)} &         & {\fs\tt(1.92e-01)}
                                          & {\fs\tt(2.60e-02)}
                                                & {\fs\tt(7.61e-05)}
\\[1ex]
\hline
\end{tabular}}
\end{table}

To check robustness of our approach with respect to the time
interval length $t$, in Table~\ref{t:3d_t1} we present results
for the coarser mesh $80\times 88\times 96$ and 
$T$ increased by a factor of 10, $T=1$.
The source vector $\bar{g}$ is exactly the same, therefore all
the tolerance values remain unchanged, only the coarse
grid error estimates are somewhat different now.
An interesting feature of the CGC approach should be observed.
Since the number of required matvecs in the regular basic method 
(the ``1 grid method'' in the table) is increased by about
45\%{} from 539 (see Table~\ref{t:3d}, line~4) to 779 matvecs, 
one can expect a similar increase
of required matvecs in the multigrid CGC algorithm.
As we see, this turns out to be not the case: the numbers of matvecs
have hardly grown.  This can be explained by a combination of 
two effects:\\
(i)~the CGC method works on the smoothed initial data,
which is profitable for the underlying Krylov subspace method;
\\
(ii)~working on a coarser mesh and with a larger $T$ means
that the parasitic eigenmodes associated with large eigenvalues
of $A$ are damped more effectively, which allows to
keep the Krylov subspace dimension bounded.

\section{Conclusions}
For moderate accuracy requirements and smooth input vectors,
the presented coarse grid correction (CGC) method allows to evaluate
the matrix exponential and $\varphi$ matrix function actions
efficiently.  Since the accuracy of the method is restricted,
upper bounds for the error are derived for the two-grid and multigrid
variants of the method.
To evaluate the coarse grid correction error in practice,
a computable error estimate is proposed.
The estimate can be computed once the coarse grid part of the
computations is carried out.
Numerical experiments demonstrate the efficiency of the approach
and its robustness with respect to space grid refinement and
to the time interval length.
\added{Our CGC approach appears to work more efficiently when combined
with Krylov subspace methods rather than with Chebyshev polynomial
iterations.  This is expected as Krylov subspace methods profit not
only from the smaller operator norm but also from the discrete structure
of the spectrum and the initial vector.}

A number of points for further research can be indicated.
First, capabilities of the approach for nonsymmetric matrices $A$
and for nonsmooth input data should be studied.
Next, it would be interesting to see whether the coarse grid
error term can be made smaller in norm by switching to a less coarse
grid.  If this is indeed possible, an adaptive CGC algorithm
with controllable accuracy can probably be designed. 
Finally, as the proposed CGC approach possesses an inherent
parallelism, its time parallel properties could be explored.
We hope to be able to address these research questions in
the future.

\medskip

\noindent
\textbf{Acknowledgments}
The author would like to thank Leonid Knizhnerman for stimulating discussions
\added{and the anonymous referee for useful suggestions to improve the paper}.

\bibliography{matfun,my_bib}
\bibliographystyle{siamplain}

\end{document}